\documentclass[11pt,a4paper]{article}
\usepackage[OT1]{fontenc}
\usepackage[english]{babel}
\usepackage{amsfonts}
\usepackage{amsthm}
\def\h{ {\cal H} }
\def\l{ {\cal L} }
\def\a{ {\cal A} }
\def\q{ {\cal Q} }
\def\b{ {\cal B} }
\def\u{ {\cal U} }

\def\ii{ {\cal I} }

\def\s{ {\cal S} }

\def\p{ {\cal P} }

\def\k{ {\cal K} }

\def\c{ {\cal C} }
\def\d{ {\cal D} }

\def\xx{ {\bf x} }

\def\dd{ {\bf d} }
\def\kk{ {\bf k} }

\newtheorem{teo}{Theorem}[section]
\newtheorem{prop}[teo]{Proposition}
\newtheorem{lem}[teo]{Lemma}

\theoremstyle{definition}
\newtheorem{rem}[teo]{Remark}

\title{A note on geodesics of projections in the Calkin algebra}
\author{Esteban Andruchow\footnote{{\sc  {Instituto de Ciencias,  Universidad Nacional de Gral. Sar\-miento,
J.M. Gutierrez 1150,  (1613) Los Polvorines, Argentina and Instituto Argentino de Matem\'atica, `Alberto P. Calder\'on', CONICET, Saavedra 15 3er. piso,
(1083) Buenos Aires, Argentina.}} e-mail: eandruch@ungs.edu.ar}}

\begin{document}

\maketitle 

\begin{abstract}
Let $\c(\h)=\b(\h)/\k(\h)$ be the Calkin algebra ($\b(\h)$ the algebra of bounded operators on the Hilbert space $\h$, $\k(\h)$ the ideal of compact operators, and $\pi:\b(\h)\to \c(\h)$ the quotient map), and $\p_{\c(\h)}$ the differentiable manifold of selfadjoint projections in $\c(\h)$. A projection $p$ in $\c(\h)$ can be lifted to a projection $P\in\b(\h)$: $\pi(P)=p$. We show that given $p,q\in\p_{\c(\h)}$, there exists a {\it minimal geodesic} of $\p_{\c(\h)}$ which joins $p$ and $q$ if and only if   there exist lifting  projections $P$ and $Q$ such that  either both $N(P-Q\pm 1)$ are finite dimensional, or  both infinite dimensional. The minimal geodesic is {\it unique} if  $p+q- 1$ has trivial anhihilator. Here the assertion that a geodesic is {\it minimal} means that it is  shorter than any other piecewise smooth curve $\gamma(t)\in\p_{\c(\h)}$, $t\in I$, joining the same endpoints, where the length of $\gamma$ is measured by $\displaystyle{\int_I \|\dot{\gamma}(t)\| d t}$.
\end{abstract}

\bigskip



\section{Introduction}
If $\a$ is a C$^*$-algebra. let $\p_\a$ denote the set of (selfadjoint) projections in $\a$. $\p_\a$ has a rich geometric structure, see for instante the papers \cite{pr} by H. Porta and L.Recht and \cite{cpr} by G. Corach, H. Porta and L. Recht. In these works, it was shown that $\p_\a$ is a differentiable (C$^\infty$) complemented submanifold of $\a_{s}$, the set of selfadjoint elements of $\a$, and has a natural linear connection, whose geodesics can be explicitly computed. A  metric is introduced, called in this context a Finsler metric: since the tangent spaces of $\p_\a$ are closed (complemented) linear subspaces of $\a_s$, they can be endowed with the norm metric. With this Finsler metric, Porta and Recht \cite{pr} showed that two projections $p,q\in\p_\a$ which satisfy that $\|p-q\|<1$ can be joined by a unique geodesic, which is minimal for the metric (i.e., it is shorter than any other smooth curve in $\p_\a$ joining the same endpoints).

In general, two projections $p,q$ in $\a$ satisfy that $\|p-q\|\le 1$, so that what  remains to consider is what happens in the extremal case $\|p-q\|=1$: under what conditions does there exist a geodesic, or a minimal geodesic, joining them. For general C$^*$-algebras, this is too vast a question. In this note we shall consider it for the case of the Calkin algebra $\a=\c(\h)=\b(\h) / \k(\h)$, where  $\b(\h)$ is the algebra of bounded linear operators in a Hilbert space  $\h$ and $\k(\h)$ is the ideal of compact operators. Denote by $\pi:\b(\h)\to \c(\h)$ the quotient $*$-homomorphism.

Let $c\in\c(\h)$ be a selfadjoint element. We say that $c$ has {\it trivial anhihilator} if $cx=0$ for $x\in\c(\h)$ implies $x=0$. Note that since $c^*=c$, this is equivalent to $xc=0$ implies $x=0$.

Clearly, if $p-q\pm 1$ have trivial anhilators, then for any lifting projecions $P$ and $Q$ (of $p$ and $q$, respectively), $\dim N(P-Q\pm 1)<\infty$: if $\dim N(P-Q\pm 1)=+\infty$, then $(P-Q-1)P_{N(P-Q-1)}=0$ so that $(p-q-1)\pi(P_{N(P-Q-1)})=0$ with $\pi(P_{N(P-Q-1)})\ne 0$ (and similarly for $P_Q+1$) . Also, these conditions are weaker than $p-q\pm 1$ being invertible.

Given $p,q\in\p_{\c(\h)}$, throughout this note we denote   $b=p+q$.  We find a  necessary and sufficient condition for the existence  of a geodesic joining $p$ and $q$ in $\p_{\c(\h)}$ ({\bf Theorem \ref{el teorema}}):
 there exists a geodesic joining $p$ and $q$ if and only if  there exist lifting projections $P$ and $Q$ of $p$ and $q$, such that $N(P-Q-1)$ and $N(P-Q+1$ are both finite dimensonal or both infinite dimensional.

Moreover, in the case of existence, also {\it minimal } a minimal geodesic exists.

With respect to {\it uniqueness} of minimal geodesics, we find a sufficient condition  described  in terms of $b$: there exists a unique minimal geodesics if $b-1$ has trivial anhihilator ({\bf Theorem \ref{unicidad}}).

Note the following elementary relation, put $a=p-q$: $a^2+b^2=2b$, so that $(b-1)^2=(1-a)(1+a)$.
Thus, if $b-1$ has trivial anhihilator, then $(b-1)^2$ also has this property, and therefore both $a\pm 1$ have this property, which shows that the uniqueness sufficient condition is stronger than the existence condition. Also note that if $\|p-q\|<1$, then $a\pm 1$ and thus also $b-1$ are invertible, so that the conditions of the above theorems are weaker than the condition $\|p-q\|<1$.

Section 2 contains preliminary facts. In Section 3  the  main results are stated. 

\section{Preliminaries}

The space $\p_\a$ is sometimes called the Grassmann manifold of $\a$. In the case when $\a=\b(\h)$, $\p_{\b(\h)}$ parametrizes the set of closed subspaces of $\h$: to each closed subspace $\s\subset\h$ corresponds the orthogonal projection $P_\s$ onto $\s$. Let us describe below the main features of the geometry of $\p_\a$ in the general case.

\subsection{Homogeneous structure}

Denote by $\u_\a=\{u\in\a: u^*u=uu^*=1\}$
 the unitary group of $\a$. It is a Banach-Lie group, whose Banach-Lie algebra is
$\a_{as}=\{x\in\a: x^*=-x\}$. This group acts on $\p_\a$ by means of $u\cdot p=upu^*, \ u\in\u_\a, \ p\in\p_\a$.
The action is smooth and locally transitive. It is known (see \cite{pr}, \cite{cpr}) that $\p_\a$ is what in differential geometry is called a {homogeneous space} of the group $\u_\a$. The local structure of $\p_\a$ is described using this action. For instance, the tangent space $(T\p_\a)_p$ of $\p_\a$ at $p$ is given by  $(T\p_\a)_p=\{x\in\a_s: x=px+xp\}$.

The isotropy subgroup of the action at $p$, i.e., the elements of $\u_\a$ which fix a given $p$, is $\ii_p=\{v\in\u_\a: vp=pv\}$.
The isotropy algebra $\mathfrak{I}_p$ at $p$ is its Banach-Lie algebra
$\mathfrak{I}_p=\{y\in\a_{as}: yp=py\}$.
 
\subsection{Reductive structure}

Given an homogeneous space, a {\it reductive } structure is a smooth distribution $p\mapsto  {\bf H}_p\subset\a_{as}$, $p\in\p_\a$, of supplements of $\mathfrak{I}_p$ in $\a_{as}$, which is invariant under the action of $\ii_p$. That is, a distribution ${\bf H}_p$ of closed linear subspaces of $\a_{as}$ verifying that ${\bf H}_p\oplus \mathfrak{I}_p=\a_{as}$; $v{\bf H}_pv^*={\bf H}_p$ for all $v\in\ii_p$;  and the map $p\mapsto {\bf H}_p$ is smooth.

In the case  of $\p_\a$, the choice of the (so called) {\it horizontal}  subspaces ${\bf H}_p$ is fairly natural, if one represents elements of $\a$ as matrices. Having fixed a base point $p\in\p_\a$, any element $a\in\a$ can be represented as a $2\times 2$ matrix.  Then, it is clear that
$$
\ii_p=\{\left(\begin{array}{cc} v_1 & 0 \\ 0 & v_2 \end{array} \right) : v_1^*v_1=v_1v_1^*=p, v_2^*v_2=v_2v_2^*=p^\perp\} \ , \ \  \mathfrak{I}_p
=\{\left(\begin{array}{cc} y_1 & 0 \\ 0 & y_2 \end{array} \right): y_i^*=-y_i\}.
$$
Thus, the natural choice for ${\bf H}_p$  defined in \cite{cpr} is  
${\bf H}_p=\{\left(\begin{array}{cc} 0 & z \\ -z^* & 0 \end{array} \right): z\in p\a p^\perp\}$.

It is worth noting that tangent vectors at $p$ have  {\it selfadjoint}  codiagonal matrices. 

As in classical differential geometry, a reductive structure on a homogeneous space defines a linear connection, and  all the invariants of the linear connection (covariant derivative, torsion and curvature tensors, etc) can be computed in terms of horizontal elements \cite{cpr}, \cite{pr}. We shall focus here on {\it geodesics}. Given a base point $p\in\p_\a$, and a tangent vector ${\bf x}=\left(\begin{array}{cc} 0 & x \\ x^*  & 0 \end{array} \right)\in (T\p_\a)_p$, the unique geodesic $\delta$ of $\p_\a$ with $\delta(0)=p$ and $\dot{\delta}(0)={\bf x}$ is given by
$$
\delta(t)= e^{tz_\xx}pe^{-tz_\xx},
$$
where $z_{\xx}:=\left(\begin{array}{cc} 0 & -x \\ x^*  & 0 \end{array} \right) $.

\subsection{Finsler metric}

As we mentioned above, one endows each tangent space $(T\p_\a)_p$ with the usual norm of $\a$. We emphasize that this (constant) distribution of norms is not a Riemannian metric (the C$^*$-norm is not given by an inner product), neither is it a Finsler metric in the classical sense (the map $a\mapsto \|a\|$ is non differentiable). Therefore the minimality result which we describe below does not follow from general considerations. It was proved in \cite{pr} using ad-hoc techniques. 
\begin{enumerate}
\item
Given $p\in\p_\a$ and $\xx\in(T\p_\a)_p$, normalized so that $\|\xx\|\le\pi/2$, then the geodesic $\delta$ remains minimal for all $t$ such that $|t|\le 1$.
\item
Given $p,q\in\p_\a$ such that $\|p-q\|<1$, there exists a unique minimal geodesic $\delta$ such that $\delta(0)=p$ and $\delta(1)=q$.
\end{enumerate}
We shall call these geodesics (with initial speed $\|\xx\|\le \pi/2$) {\it normalized} geodesics.
\begin{rem}
One does not have to deal with homogeneous reductive spaces in order to understand (at least partially) these results. Using the C$^*$-norm at every tangent space, means that a curve $\delta$ as above (never mind calling it a geodesic) has the following property: for every piecewise differentiable curve $\gamma(t)\in\p_\a$ whose endpoints are $\gamma(t_0)=p$ and $\gamma(t_1)=q$, it holds that
$$
\displaystyle{\int_{t_0}^{t_1}\|\dot{\gamma}(t)\| d t \ge \int_{0}^{1}\|\dot{\delta}(t)\| d t=\|\xx\|}.
$$
\end{rem}

\subsection{Projections in $\b(\h)$}
Let us finish this preliminary section by recalling the case $\a=\b(\h)$. Given $P,Q\in\p_{\b(\h)}$ (operators in Hilbert space will be denoted with upper case letters, this inconsistency will prove useful later, when we deal with the Calkin algebra), in \cite{jmaa} it was proved that there exists a geodesic of the linear connection just described, which joins $P$ and $Q$, if and only if $\dim R(P)\cap N(Q)=\dim N(P)\cap R(Q)$.

The geodesic  can be chosen minimal. There is a {\it unique} minimal geodesic joining these points if and only if these dimensions equal zero.

Note the fact that the condition $\|P-Q\|<1$ implies that these dimensions are zero, but the converse is not true.

\section{Projections in the Calkin algebra}

From now on, we consider  the case of $\a=\c(\h)$.  Tipically, upper case letters $P, A, X$ will denote elements of $\b(\h)$ and their lower case counterparts $p=\pi(P), a=\pi(A), x=\pi(X)$ the corresponding elements of $\c(\h)$.

Let us point out the following elementary facts:
\begin{lem}
Let $p\in\p_{\c(\h)}$, $p\ne 0,1$. Then there exists $P\in\p_{\b(\h)}$ such that $\pi(P)=p$
\end{lem}
\begin{proof}
Let $T\in\b(\h)$ such that $\pi(T)=p$. Since $p^*=p$, $\pi(\frac12(T+T^*))=p$, i.e., we can suppose that $T^*=T$. Clearly $T^2-T\in\k(\h)$, and thus the spectrum of $T$ accumulates only (eventualy) at $0$ and $1$.  Let $\varphi:\mathbb{R}\to\mathbb{R}$ be a continuous function which is equal to $1$ on an interval $I$ containing $1$ in its interior,  $0\notin I$, such that $\varphi$ is  zero in $\sigma(T)\cap (\mathbb{R}\setminus I)$ (note that the spectrum $\sigma(T)$ is countable). Then $P=\varphi(T)$ is a selfadjoint projection, and 
$$
\pi(P)=\varphi(\pi(T))=\varphi(p)=p,
$$
\end{proof}
\begin{rem}
Note that any pair of proper projections  $p,q$  ($\ne 0,1$) in $\c(\h)$ are unitarilly equivalent. Indeed, let $P, Q\in\p_{\b(\h)}$ such that $\pi(P)=p$ and $\pi(Q)=q$. Thus, $P, P^\perp, Q$ and $Q^\perp$ have infinite rank ($p,q\ne 0,1$). Then $P$ and $Q$ are unitarilly equivalent: there exists $U\in\u_{\b(\h)}$ such that $UPU^*=Q$. Then $u=\pi(U)$ is a unitary element in $\c(\h)$ such that $upu^*=q$.
\end{rem}

If $\Delta(t)=e^{tZ}Pe^{-tZ}$ is a geodesic in $\p_{\b(\h)}$ with $\|Z\|\le \pi/2$, then it is clear that $\delta(t)=\pi(\Delta(t))$ is a geodesic in $\p_{\c(\h)}$ with $\|z\|\le \pi/2$. Indeed,
$$
\delta(t)=e^{tz}pe^{-tz},
$$
where $\pi(P)=p$, and $z=\pi(Z)$ satifies $z^*=-z$ and 
$$
pzp=p^\perp zp^\perp=\pi(PZP)=\pi(P^\perp ZP^\perp)=0.
$$
Moreover, $\|z\|\le \|Z\|\le \pi/2$.
The next result shows that there is a converse for this statement: any geodesic $\delta$ of $\p_{\c(\h)}$ with initial speed $\|z\|\le \pi/2$ lifts to a geodesic $\Delta$ with initial speed $\|Z\|\le\pi/2$. It is based on the following elementary observation,  which is an excercise and certainly  well known. We include a proof.
\begin{lem}
Let $x=x^*\in\c(\h)$. Then there exists $X=X^*\in\b(\h)$ such that $\pi(X)=x$ and $\|X\|=\|x\|$.
\end{lem}
\begin{proof}
Clearly, there exists $X_0=X_0^*\in\b(\h)$ such that $\pi(X_0)=x$. Recall the Weyl-von Neumann theorem (see for instance \cite{libro davidson}), which states that there exists a diagonalizable selfadjoint operator $X_d$ and a compact operator $K$ such that $X_0=X_d+K$. Thus, we may suppose that $X_0$ is diagonalizable, and let us fix the orthonormal basis of $\h$ which diagonalizes $X_0$. Denote by $\dd=\{d_n\}$ the sequence of entries of $X_0$. It suffices to find a sequence $\kk\in c_0$ (the space of sequences which converge to zero) such that 
$$
\|\dd+\kk\|_\infty\le\ \|\dd+\kk'\|_\infty
$$
for any other sequence $\kk'\in c_0$. Indeed, denote by $\d:\b(\h)\to\b(\h)$ the linear positive contraction which assigns to $T$ the diagonal operator $\d(T)$ with the same diagonal entries as $T$. Then, if $K_0$ denotes the diagonal compact operator with $\kk$ in its diagonal, and $K'$ is any other compact operator (with diagonal $\kk'$)
$$
\|X_0+K_0\|=\|\dd+\kk\|_\infty\le\|\dd+\kk'\|_\infty=\|\d(X_0+K')\|\le \|X_0+K'\|.
$$
In order to find an optimal sequence $\kk\in c_0$ such that $\|\dd+\kk\|_\infty$ is as small as possible, note that 
$$
\inf\{\|\dd+\kk'\|_\infty: \kk'\in c_0\}=\limsup |\dd|.
$$
Indeed, on one hand, there can only be finitely many $d_n$ such that $d_n<\limsup \dd$, so that $\inf\{\|\dd+\kk'\|_\infty: \kk'\in c_0\}\le\limsup |\dd|$. On the other hand, 
for any $\kk'=\{k'_n\}\in c_0$,
$$
\limsup |\dd|=\limsup |\dd+\kk'|\le \sup_n |d_n+k'_n|=\|\dd+\kk'\|_\infty.
$$
Next, note that there exists an optimal $\kk\in c_0$ such that $\|\dd+\kk\|_\infty=\limsup |\dd|$. Clearly it suffices to reason with sequences of non negative numbers. Define
$$
\dd_0=\left\{ \begin{array}{lll} \limsup \dd & \hbox{ if } & d_n>\limsup \dd \\
d_n  & \hbox{ if } &  d_n\le \limsup \dd \end{array} \right.
$$
and 
$$
\kk_0=\left\{ \begin{array}{lll} d_n-\limsup \dd & \hbox{ if } & d_n>\limsup \dd \\
0  & \hbox{ if } &  d_n\le \limsup \dd \end{array} \right. .
$$
Clearly $\dd+\kk=\dd_0+\kk_0$ and $\|\dd_0\|_\infty=\limsup \dd$. Finally, $\kk_0\in c_0$: any subsequence of of $\kk_0$ has a converging sebsequence, which can only converge to zero.
\end{proof}
\begin{prop}\label{levanta geodesicas}
Let $\delta(t)=e^{tz}pe^{-tz}$ be a (normalized) geodesic of $\p_{\c(\h)}$: $z^*=-z$, $pzp=p^\perp zp^\perp=0$ and $\|z\|\le\pi/2$.
Then there exists a normalized geodesic $\Delta(t)=e^{tZ}Pe^{-tZ}$ which lifts $\delta$: $Z^*=-Z$, $PZP=P^\perp Z P^\perp=0$, $\|Z\|\le \pi/2$, $\pi(P)=p$ and $\pi(Z)=z$
\end{prop}
\begin{proof}
Given $z^*=-z$, by the above lemma ($i z$ is selfadjoint), there exists $Z_0\in\b(\h)$ such that $Z_0^*=-Z_0$, $\pi(Z_0)=z$ and $\|Z_0\|=\|z\|$. Pick $P\in\p_{\b(\h)}$ such that $\pi(P)=p$. Consider 
$$
Z=PZ_0P^\perp +P^\perp Z_0 P.
$$
Clearly $\pi(Z)=pzp^\perp+p^\perp z p=z$, because $z$ is $p$-codiagonal. Also, $Z$ is $P$-codiagonal, and 
$$
Z^*=(PZ_0P^\perp)^*+(P^\perp Z_0 P)^*=-P^\perp Z_0P-PZ_0P^\perp=-Z.
$$
In matrix form, $Z=\left( \begin{array}{cc} 0 & PZ_0P^\perp \\ P^\perp Z_0 P & 0\end{array} \right)$, so that 
$$
\|Z\|=\|PZ_0P^\perp\|=\|P^\perp Z_0P\|\le \|Z_0\|.
$$
Since $\|Z_0\|=\|z\|\ge \|Z\|$, equality holds.
\end{proof}
\begin{rem}\label{punto inicial}
Note that when lifting a geodesic $\delta$ of $\p_{\c(\h)}$  to a geodesic $\Delta$ of $\p_{\b(\h)}$,  the initial point $\Delta(0)=P$ can be chosen to be any projection in the fiber of $\pi^{-1}(\delta(0))$ of $\delta(0)$.
\end{rem}
Our main result follows.

\begin{teo}\label{el teorema}
Let $p,q$ be proper projections in $\c(\h)$ (i.e., $p,q\ne 0,1$).  Then there exists a  geodesic of $\p_{\c(\h)}$ joining $p$ and $q$ if and only there exist liftings projections $P$ and $Q$ of $p$ and $q$ ($\pi(P)=p$, $\pi(Q)=q$) such that $N(P-Q-1)$ and $N(P-Q+1)$ are  both finite dimensional, or else, are both infinite dimensional.

In either case, the geodesic can be chosen minimal.
\end{teo}
\begin{proof}
Suppose first that there exists a geodesic $\delta(t)=e^{tz}pe^{-tz}$ such that $\delta(1)=q$. Then $\delta$ lifts to $\Delta(t)=e^{tZ}Pe^{-tZ}$. Denote $Q=\Delta(1)$. The existence of $\Delta$ implies that (see\cite{jmaa})
$\dim R(P)\cap N(Q)=\dim N(P)\cap R(Q)$.  It is easy to see that
$$
R(P)\cap N(Q)=\{\xi\in \h: P\xi=\xi \hbox{ and } Q\xi=0\}=\{\xi\in\h: (P-Q)\xi=\xi\}=N(P-Q-1).
$$
Similarly, $N(P)\cap R(Q)=N(P-Q+1)$. 
Thus the necessary part is clear.

Conversely, let $P,Q$ be projections in $\b(\h)$ such that $\pi(P)=p$ and $\pi(Q)=q$.  If  $P-Q+1$ and $P-Q-1$  have finite dimensional nullspaces. Let us consider the following $5$-space decomposition of $\h$ which is customary when dealing with two subspaces (\cite{dixmier}, \cite{halmos}):
$$
\h=R(P)\cap R(Q) \ \oplus \ N(P)\cap N(Q) \ \oplus \ R(P)\cap N(Q) \ \oplus \ N(P)\cap R(Q) \ \oplus \ \h_0,
$$
where $\h_0$, the orthogonal complement of the sum of the first four, is usually called the {\it generic part} of $P$ and $Q$. It is known, that these subspaces reduce both $P$ and $Q$, and that the reduction of our problem to the generic part has a positive solution: on the generic part, there exists a unique (minimal) geodesic joining these projections (see \cite{jmaa}). Since  we are currently supposing that $R(P)\cap N(Q)$ and $N(P)\cap R(Q)$ are finite dimensional, we can replace $P$ and $Q$, which in this decomposition are
$$
P=1\oplus 0 \oplus 1 \oplus 0 \oplus P_0 \  \hbox{ and  } \ Q=1\oplus 0 \oplus 0 \oplus 1 \oplus Q_0
$$
with $P',Q'$ given by
$$
P'=1\oplus 0 \oplus 0 \oplus 0 \oplus P_0 \  \hbox{ and  } \ Q'=1\oplus 0 \oplus 0 \oplus 0 \oplus Q_0,
$$
which are indeed projections, and differ from $P,Q$ on a finite dimensional subspace. Thus $\pi(P')=p$ and $\pi(Q')=q$. Note that
$$
R(P')\cap N(Q')=N(P')\cap R(Q')=\{0\}.
$$
Therefore there exists a geodesic $\Delta(t)=e^{tZ}P'e^{-tZ}$ joining $P'$ and $Q'$ in $\p_{\b(\h)}$. Thus $\delta=\pi(\Delta)$ is a geodesic in $\p_{\c(\h)}$ joining $p$ and $q$.

If  $R(P)\cap N(Q)$ and $N(P)\cap R(Q)$ are infinite dimensional,   there exists a normalized geodesic of $\p(\h)$ joining $P$ and $Q$, and thus the claim follows similarly in this case.

In either case, the geodesics are given by exponents $z$ or $Z$ with norm less or equal than $\pi/2$, so that they are minimal.
\end{proof}

For the uniqueness of normalized geodesics, we have the following sufficient condition in terms of $b=p+q$:

\begin{teo}\label{unicidad}
If $b-1$ has trivial anhililator, then there exists a unique minimal geodesic between $p,q\in\p_{\c(\h)}$.
\end{teo}
\begin{proof}

Suppose that $b-1$ has trivial anhihilator. Then, as remarked in the Introduction, 
$\dim N(P-Q\pm 1)<\infty$.  Let $\delta$ and $\delta'$ be two normalized minimal geodesics joining $p$ and $q$: $\delta(t)=e^{tz}pe^{-tz}$, $\delta'(t)=e^{tz'}pe^{-tz'}$, $\|z\|=\|z'\|\le \pi/2$, with $z,z'$ $p$-codiagonal, and $\delta(1)=\delta'(1)=q$. Then, as in the proof of Theorem \ref{el teorema}, there exist projections $P$ and $Q$, relative to the geodesic $\delta$, such that $R(P)\cap N(Q)=N(P)\cap R(Q)=\{0\}$, and $Z^*=-Z$ $P$-codiagonal with $\|Z\|=\|z\|$ such that 
$\Delta(t)=e^{tZ}Pe^{-tZ}$ lifts $\delta$. Denote $Q=\Delta(1)$.  Similarly, let $P',Q',Z'$ the data corresponding to $\delta'$, with $R(P')\cap N(Q')=N(P')\cap R(Q')=\{0\}$. Note that $P-P'$ and $Q-Q'$ are compact. The exponents $Z$ and $Z'$ are uniquely determined by $P,Q$ and $P',Q'$ respectively. Let us recall from \cite{jmaa} how they are constructed. The operator $B-1=P+Q-1$ is selfadjoint and has trivial nullspace. Indeed, since $B-1=P-Q^\perp$ is a difference of projections, its nullspace is trivial
$$
N(B-1)=R(P)\cap R(Q^\perp)\oplus N(P)\cap N(Q^\perp)=R(P)\cap N(Q)\oplus N(Q)\cap R(P)=\{0\}.  
$$
Then, if $B-1=V|B-1|$ is the polar decomposition, $|B-1|$ has trivial nullspace and $V$ is a symmetry ($V^*=V^{-1}=V$). Then, it was shown in \cite{jmaa} that
\begin{equation}\label{la posta}
e^{Z}=V(2P-1).
\end{equation}
In fact, this formula was shown for projections in generic position. In the case at hand, we have also to take the subspace $R(P)\cap R(Q) \ \oplus \ N(P)\cap N(Q)$ into account. On this subspace, $P$ and $Q$ coincide, so that $P+Q-1$ equals $2P-1$, thus $|P+Q-1|$ is the identity and therefore $V$ coincides with $2P-1$. On the other hand, since $P$ and $Q$ coincide here, the exponent $Z$ of the minimal geodesic is $0$. Thus, on $R(P)\cap R(Q) \oplus N(P)\cap N(Q)$, equation (\ref{la posta}) is $e^0=(2P-1)(2P-1)=1$, which is trivial.

Equation (\ref{la posta}) determines $Z$, as the unique  anti-Hermitian logarithm of $V(2P-1)$ with spectrum in the interval $[-\pi/2,\pi/2]$. Analogously, we have that
$e^{iZ'}=V'(2P'-1)$, where $V'$ is the unitary part in the polar decompostion of $B'-1=P'+Q'-1$. Thus, in order to show that $z=z'$, it suffices to show that $e^{z}=e^{z'}$, i.e., that $\pi(V')(2p-1)=\pi(V')(2p-1)$. Since $2p-1$ is  invertible, we have to show that $\pi(V)=\pi(V')$. 
Note that $b=\pi(B)=\pi(B')$, and since $\pi$ is a $*$-homomorphism, $\pi(|B-1|)=|\pi(B-1)|=|b-1|$. Thus $b-1=\pi(V)|b-1|=\pi(V')|b-1|$, or equivalently, $\pi(V)(b-1)=\pi(V')(b-1)$. Then $(\pi(V)-\pi(V'))(b-1)=0$, and since $b-1$ has trivial anhihilator, $\pi(V)=\pi(V')$.
\end{proof}

\begin{prop}
Let $p,q\in\p_{\c(\h)}$. Suppose that there exist lifting projections $P$ and $Q$ such that $\dim N(P-Q\pm 1)=+\infty$. Then there exist infinitely many minimal geodesics joining $p$ and $q$ in $\p_{\c(\h)}$.
\end{prop}
\begin{proof}
Let $P,Q\in\p_{\b(\h)}$ such that $\pi(P)=p$ and $\pi(Q)=q$. Then $R(P)\cap N(Q)$ and $N(P)\cap R(Q)$ are infinite dimensional. Then there exist isometries 
$$
V, V': N(P)\cap R(Q) \to R(P)\cap N(Q)
$$
such that $V-V'$ is not compact. For instance, let $V$ be  such an isometry, and pick a unitary $U$ acting in $R(P)\cap N(Q)$ such that $U-1$ is not compact, and put $V'=UV$. There are infinitely many unitaries $U$ with this property, such that $\pi(U)$ are different: the unitary group of $\c(R(P)\cap N(Q))$ is infinite. Then (see for instance \cite{jmaa}), one constructs geodesics between $P$ and $Q$ determining its velocity vectors as follows:
$$
Z=0\oplus 0 \oplus \{i\pi/2(V+V^*)\} \oplus Z_0,
$$
in the (four space) decomposition 
$$
\h=R(P)\cap R(Q) \  \oplus N(P)\cap N(Q) \ \oplus \{R(P)\cap N(Q) \ \oplus \  N(P)\cap  R(Q)\} \ \oplus \h_0,
$$
and $Z'$ defined analogously, with $V'$. The part $Z_0$ acting in $\h_0$ ($\|Z_0\|\le \pi/2$) is uniquely determined. Then clearly $Z-Z'$ is non compact. It follows that if $z=\pi(Z)$ and $z'=\pi(Z')$, then $\delta(t)=e^{tz}pe^{-tz}$ and $\delta'(t)=e^{tz'}pe^{-tz'}$ are two geodesics joining $p$ and $q$, with
$$
\dot{\delta}(0)-\dot{\delta'}(0)=z-z'\ne 0,
$$
i.e., $\delta\ne\delta'$.
\end{proof}

\end{document}